\documentclass[a4paper,10pt]{amsart}
 \usepackage{amsfonts,mathrsfs}
 \usepackage{amsthm}
 \usepackage{amssymb}
 \usepackage{enumerate}
 \usepackage{tikz-cd}
\theoremstyle{definition}
\newtheorem{Def}{Definition}[section]
\newtheorem{ex}[Def]{Example}
\newtheorem{rem}[Def]{Remark}

\theoremstyle{plain}
\newtheorem{prop}[Def]{Proposition}
\newtheorem{thm}[Def]{Theorem}
\newtheorem*{thm*}{Theorem}
\newtheorem{lem}[Def]{Lemma}
\newtheorem{cor}[Def]{Corollary}
\newtheorem*{cor*}{Corollary}

\newtheorem*{con*}{Conjecture}
\newtheorem*{frag*}{Question}
\newtheorem*{verm*}{Vermutung}

\usepackage{mathrsfs}

\newcommand{\Pic}{\operatorname{Pic}}

\newcommand{\conv}{\operatorname{conv}}

\newcommand{\Ho}{\operatorname{H}} 

\newcommand{\diag}{\operatorname{diag}}

\newcommand{\C}{{\mathbb C}}
\newcommand{\R}{{\mathbb R}}
\newcommand{\pp}{\mathbb{P}}

\newcommand{\N}{{\mathbb N}}
\newcommand{\Z}{{\mathbb Z}}

\newcommand{\F}{{\mathbb F}}

\title{Totally real theta characteristics}
\author{Mario Kummer}
\address{Technische Universit\"at, Berlin, Germany} 
\email{kummer@tu-berlin.de}
\thanks{Corresponding author:\\ Mario Kummer\\Technische Universit\"at Berlin\\
Institut f\"ur Mathematik\\
Sekretariat MA 3-2\\
Stra\ss{}e des 17. Juni 136\\
D-10623 Berlin\\
Phone: +49-30-314-27602\\
E-Mail:	kummer@tu-berlin.de}
\newcommand{\comment}[1]{}

\begin{document}


\begin{abstract}
 A totally real theta characteristic of a real curve is a theta characteristic which is linearly equivalent to a sum of only real points. These are closely related to the facets of the convex hull of the canonical embedding of the curve. We give upper and lower bounds for the number of both of totally real theta characteristics and facets.
\end{abstract}
\maketitle

\section{Introduction}
Let $X$ always denote a smooth projective and geometrically irreducible curve over $\R$ and consider its Picard group $\Pic(X)$. We furthermore assume that the real part $X(\R)$ is nonempty. A \textit{theta characteristic} is a divisor class $M\in\Pic(X)$ such that $2M$ is the canonical class. Depending on the parity of $l(M)$ it is called \textit{odd} or \textit{even}. It is classically known that the number of theta characteristics of the complexification $X_\C$ of $X$ is precisely $2^{2g}$ where $g$ is the genus of the curve. Among those there are $2^{g-1}(2^{g}+1)$ even and $2^{g-1}(2^{g}-1)$ odd theta characteristics. The number of real theta characteristics is also determined by the discrete invariants of $X$. Namely, we let $s$ denote the number of connected components of the real points $X(\R)$ and we let $a=1$ if $X(\C)\setminus X(\R)$ is connected and $a=0$ otherwise. In the latter case we say that $X$ is of \textit{dividing type} since the real part divides the complex locus. The number of (real) even and odd theta characteristics in $\Pic(X)$ equals $2^{g-1}(2^{s-1}+1-a)$ and $2^{g-1}(2^{s-1}-1+a)$ respectively \cite[Prop. 5.1]{grossharris}.

Let $D$ be some effective divisor on $X$, say $\sum_{i=1}^r P_i+\sum_{j=1}^t Q_j$ where the $P_i$ have degree one and the $Q_j$ degree two. The $P_i$ correspond to points from $X(\R)$ whereas the $Q_j$ correspond to pairs of complex conjugate points from $X(\C)\setminus X(\R)$. We say that the divisor $D$ is \textit{totally real} if $D$ is the sum of real points only, i.e. $t=0$. We say that a divisor class is totally real if it contains an effective totally real divisor. Note that by \cite[Cor. 2.10]{scheider} every divisor class of sufficiently high degree is totally real. We are however interested in the number of totally real (odd) theta characteristics of $X$. This quantity does not only depend on the discrete invariants $g$, $s$ and $a$ anymore, so our goal is establish some bounds.

If we look at the canonical embedding of $X$ to $\pp^{g-1}$, then effective theta characteristics give rise to \textit{contact hyperplanes}, i.e. hyperplanes that intersect $X$ of even order only. If the theta characteristic is totally real, we get a hyperplane that touches the curve only in real points. These were studied in \cite{emch, corey, kulkarni2017real} for curves of genus four. In \cite[Question 2, page 254]{kulkarni2017real} they ask for lower bounds on the number of such totally real hyperplanes. We establish such bounds that are sharp for $g=3$ and very close to being sharp for $g=4$ according to \cite[Table 1]{kulkarni2017real}.

Now let $C\subset\R^{g-1}$ be a compact algebraic set whose Zariski closure in $\pp^{g-1}$ is a canonical curve $X$ as above such that $X(\R)=C$. We consider the convex hull $K=\conv(C)$ of $C$. Each $(g-2)$-face of $K$ corresponds to a unique totally real theta characteristic $M$ with $l(M)=1$. Conversely, if we have a totally real theta characteristic with exactly one global section, the corresponding hyperplane in $\pp^{g-1}$ will cut out a $(g-2)$-face of the convex hull of $X(\R)$ in a suitable affine chart. Using this correspondence we show that the convex hull $K$ has at most $2^{g-1}$ faces of dimension $g-2$. For $g=4$ this upper bound is sharp and therefore yields an answer to \cite[Question 11, page 254]{kulkarni2017real}. On the other hand we also establish lower bounds on the number of $(g-2)$-faces. Both upper and lower bounds for the number of $(g-2)$-edges are tight in the cases $g=3,4$. Summing up over all essentially different affine charts, we obtain a lower bound for totally real odd theta characteristics as well. For example if $s=g+1$ this lower bound is $\binom{g+1}{g-1}\cdot 2^{g-1}$.

\section{Real curves and their differentials}
\subsection{Real theta characteristics}
This section is to recall some facts on the Jacobian of a real curve, see \cite{grossharris} and \cite{Vin93}.
Let $X$ always be a smooth geometrically irreducible projective real curve of genus $g$ with $X(\R)$ having $s>0$ connected components. The case $X(\R)=\emptyset$ is not interesting for our purposes. The complex conjugation on $X(\C)$ induces an involution on the integral homology $\Ho_1(X(\C),\Z)$ that we denote by $\tau$. After a suitable choice of a symplectic homology basis $A_1,\ldots,A_g,B_1,\ldots,B_g$, the representing matrix of $\tau$ is given by $$T=\begin{pmatrix}I&0\\H&-I\end{pmatrix}$$ where $H$ is a (block)-diagonal matrix of rank $r=g+1-s$. More precisely, if $X$ is of dividing type, then letting $H_0=\begin{pmatrix}0&1\\1&0\end{pmatrix}$ we have $$H=\begin{pmatrix}H_0&&&&&\\&\ddots&&&&\\&&H_0&&&\\&&&0&&\\&&&&\ddots&\\&&&&&0\end{pmatrix}$$ and $H=\begin{pmatrix}I&0\\0&0\end{pmatrix}$ otherwise \cite[Prop. 2.2]{Vin93}, compare \cite{grossharris}. Now we let $\omega_1,\ldots,\omega_g\in\Gamma(X,\Omega_{X/\R})$ be a basis of the space of real differentials such that $\int_{A_j}\omega_i=\delta_{ij}$ for all $1\leq i,j\leq g$. Then the Jacobian variety of $X$ is given by $\C^g/\Lambda$ where $\Lambda$ is the lattice generated by the unit vectors $e_1,\ldots,e_g$ and the columns $z_1,\ldots,z_g$ of the period matrix $Z=(\int_{B_j}\omega_i)$. Using this basis we represent elements of $\C^g/\Lambda$ by column vectors of length $2g$ having real entries. The Abel--Jacobi map gives an isomorphism $\mu: \Pic^0(X_\C)\to\C^g/\Lambda$ and with the above identification we have $\mu(\overline{D})=T^t\mu(D)$. The choice of the symplectic homology basis determines the \textit{Riemann constant} $M\in\Pic^{g-1}(X_\C)$ which is an even theta characteristic (not necessarily real) satisfying $$\mu(\overline{M}-(g-1)P)=\mu(M-(g-1)P)+\frac{1}{2}\binom{\diag(H)}{0}$$ where $P\in X(\R)$ is some real point of $X$ \cite[Prop. 2.2]{Vin93}. This implies that a divisor $N\in\Pic^{g-1}(X_\C)$ is real if and only if $$\mu(N-M)=T^t\mu(N-M)+\frac{1}{2}\binom{\diag(H)}{0}.$$


Via subtraction by $M$ one can identify theta characteristics with $2$-torsion points in the Jacobian. Then the theta characteristics of $X_\C$ correspond to the elements $\binom{u}{v}\in\F_2^{2g}$. The even (odd) theta characteristics correspond to those vectors with $u^t v=0$ ($1$ respectively). The above discussion implies that such a theta characteristic is fixed by the involution if and only if $Hv=\diag(H)$. This shows that there are $2^{g+s-1}$ real theta characteristics. If $X$ is of dividing type, then there are $2^{g-1}(2^{s-1}-1)$ real odd and $2^{g-1}(2^{s-1}+1)$ real even theta characteristics. Otherwise there are $2^{g+s-2}$ theta characteristics of each parity, compare \cite[Prop. 5.1]{grossharris}.

\subsection{Real definite differentials}
Here we recall the content of \cite[\S 4]{Vin93}. Let $X_0,\ldots,X_{s-1}$ be the connected components of $X(\R)$. We fix some orientation on $X(\R)$. If $X$ is of dividing type, we choose the \textit{complex orientation}, i.e. the orientation induced on $X(\R)$ by an orientation on one of the connected components of $X(\C)\setminus X(\R)$. This is well-defined up to global reversing. Let $\omega\in\Gamma(X,\Omega_{X/\R})$ and let $p\in X(\R)$ not a zero of $\omega$. Taking a real local coordinate in $p$ that agrees with the chosen orientation on $X(\R)$, we obtain a well-defined sign of $\omega$ at $p$. If $\omega$ is a definite differential, i.e. has all real zeros of even order, then it has constant sign $\sigma_i(\omega)$ on each component $X_i$. The following result characterizes the possible sign patterns a definite differential can have.

\begin{prop}[Last paragraph on page 469 of \cite{Vin93}]\label{prop:vincomp}
 For any choice of signs $\epsilon_1,\ldots,\epsilon_{s-1}\in\{\pm 1\}$ there is a definite differential $\omega\in\Gamma(X,\Omega_{X/\R})$ with $\sigma_0(\omega)=1$ that satisfies $\sigma_i(\omega)=\epsilon_i$ for $i=1,\ldots,s-1$, except when $X$ is of dividing type and all $\epsilon_i=1$. In that case there is no such differential.
\end{prop}

\begin{figure}[h]
 \includegraphics[width=5cm]{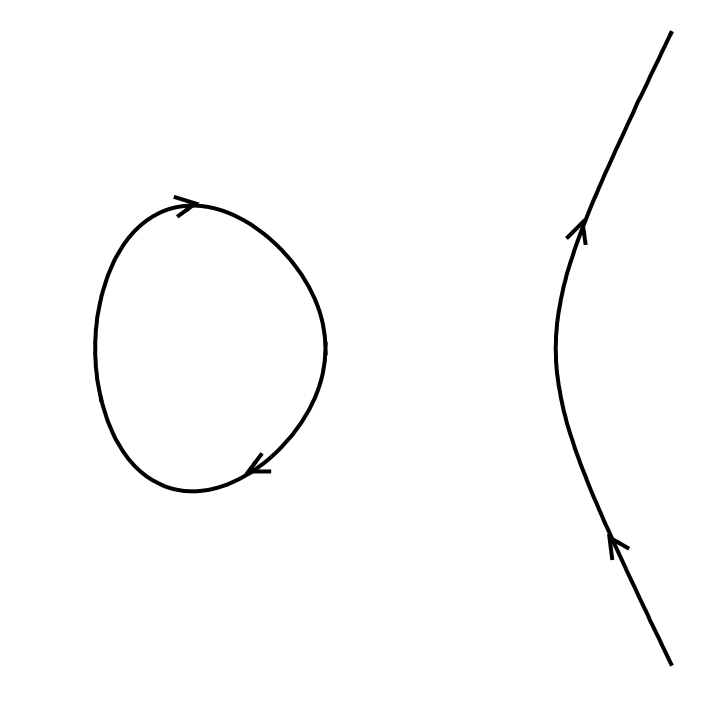} 
\caption{The planar cubic curve defined by $y^2=x^3-x$ and the orientation induced by the differential $\frac{dx}{y}$.  }
\label{fig:elliptic}
\end{figure}

We will need the following slightly refined version of that result stating that we can even find such differentials without any real zeros.

\begin{cor}\label{cor:avoidloc}
 For any choice of signs $\epsilon_1,\ldots,\epsilon_{s-1}\in\{\pm 1\}$ there is a definite differential $\omega\in\Gamma(X,\Omega_{X/\R})$ with $\sigma_0(\omega)=1$ having no real zeros that satisfies $\sigma_i(\omega)=\epsilon_i$ for $i=1,\ldots,s-1$, except when $X$ is of dividing type and all $\epsilon_i=1$. In that case there is no such differential.
\end{cor}

\begin{proof}
 Consider two definite differentials $\omega_1,\omega_2\in\Gamma(X,\Omega_{X/\R})$ with the same sign pattern $\epsilon$. Then $\omega=\omega_1+\omega_2$ is also a definite differential with the sign pattern $\epsilon$. Moreover, the real zeros of $\omega$ are exactly the common real zeros of $\omega_1$ and $\omega_2$. Therefore, it suffices to show that there is no point $P_0\in X$ which is a common zero of all definite differentials having the sign pattern $\epsilon$. By \cite[Thm. 3.1]{Vin93} there is a $(g-1)$-dimensional real analytic subvariety $V$ of $\Pic^{g-1}(X_\C)$ with the following property: Each divisor class from $V$ contains an effective divisor $D$ such that $D+\overline{D}$ is the zero divisor of a definite differential $\omega_D\in\Gamma(X,\Omega_{X/\R})$ whose sign pattern agrees with $\epsilon$ \cite[Thm. 4.1]{Vin93}. Now since there are only finitely many possibilities to write the zero divisor of a definite differential as $D+\overline{D}$ for some effective divisor $D$, we find that the set of definite differentials with sign pattern $\epsilon$ must be full dimensional in the real vector space $\Gamma(X,\Omega_{X/\R})$. In particular, they do not have a common zero.
\end{proof}

Finally, the following theorem in combination with the discussion from the previous section allows us to calculate the number of odd theta characteristics giving rise to definite differentials with prescribed sign.

\begin{thm}[Thm. 4.1 and Prop. 4.2 in \cite{Vin93}]
 There are signs $\sigma_1,\ldots,\sigma_{s-1}$ such that the following holds: Let $\omega$ be a real definite differential in $\Gamma(X,\Omega_{X/\R})$ with $\sigma_0(\omega)=1$ and $\textnormal{div}(\omega)=2D$ for some theta characteristic $D$ corresponding to $\binom{u}{v}\in\F_2^{2g}$. Then we have $\sigma_i(\omega)=(-1)^{u_{r+i}}\sigma_i$ for $i=1,\ldots,s-1$ (recall that $r=g+1-s$). If $X$ is of dividing type, then $\sigma_i=1$ for all $i=1,\ldots,s-1$.
\end{thm}


\begin{cor}\label{cor:countodd}
 For any choice of signs $\epsilon_1,\ldots,\epsilon_{s-1}\in\{\pm 1\}$ there are $2^{g-1}$ odd theta characteristics $D$ such that a differential $\omega\in\Gamma(X,\Omega_{X/\R})$ with $\textnormal{div}(\omega)=2D$ and $\sigma_0(\omega)=1$ satisfies $\sigma_i(\omega)=\epsilon_i$ for $i=1,\ldots,s-1$, except when $X$ is of dividing type and all $\epsilon_i=1$. In that case there are no such odd theta characteristics.
\end{cor}

\begin{proof}
 This is just counting vectors $\binom{u}{v}\in\F_2^{2g}$ such that $u^t v=1$ and $Hv=\diag(H)$ where the last $s-1$ entries of $u$ are fixed.
\end{proof}

%
%
%

\section{Some convex geometry}

Before we prove our results we need some preparation from convex geometry. Morally speaking, we will prove the $n$-dimensional analog of the folklore fact that a three-legged stool is stable (if designed properly).

\begin{lem}
 Let $K_1,\ldots,K_m\subset\R^n$ be some convex bodies, i.e. compact with nonempty interior. The following are equivalent:
 \begin{enumerate}[(i)]
  \item For each subset $I\subset[m]:=\{1,\ldots,m\}$, the convex hulls $\conv(\bigcup_{i\in I} K_i)$ and $\conv(\bigcup_{i\in [m]\setminus I} K_i)$ are disjoint.
  \item Each choice of points $x_1,\ldots,x_m$ with $x_i\in K_i$ is affinely independent.
 \end{enumerate}
 If $(i)$ and $(ii)$ are satisfied, we say that $K_1,\ldots, K_m$ are \textnormal{strongly separated}.
\end{lem}

\begin{proof}
We use the same type of argument as it has been used in \cite{Radon}.

 ``$\neg(i)\Rightarrow\neg(ii)$'': If there is an $x\in \conv(\bigcup_{i\in I} K_i)\cap\conv(\bigcup_{i\in [m]\setminus I} K_i)$, then $$\sum_{i\in I}\lambda_i x_i=x=\sum_{i\in[m]\setminus I}\lambda_i x_i$$ for some $x_i\in K_i$ and $\lambda_i\geq0$ such that $$\sum_{i\in I}\lambda_i=1=\sum_{i\in[m]\setminus I}\lambda_i.$$ But then $x_1,\ldots,x_m$ are affinely dependent since $\sum_{i\in I}\lambda_ix_i+\sum_{i\in[m]\setminus I}(-\lambda_i)x_i=0$, $\sum_{i\in I}\lambda_i+\sum_{i\in[m]\setminus I}(-\lambda_i)=0$ and $\lambda_i\neq0$ for at least two indices $i$.
 
``$\neg(ii)\Rightarrow\neg(i)$'': Let $x_1,\ldots,x_m$ be affinely dependent points of $\R^n$ with $x_i\in K_i$. Then $\sum_{i=1}^m\lambda_ix_i=0$ for some $\lambda_i\in\R$ with $\sum_{i=1}^m\lambda_i=0$ and $\lambda_i\neq0$ for at least one index $i$. Let $I$ be the (nonempty) set of all $i\in[m]$ with $\lambda_i>0$, let $\Lambda:=\sum_{i\in I}\lambda_i>0$, and let $x:=\sum_{i\in I}\frac{\lambda_i}{\Lambda}x_i\in\conv(\bigcup_{i\in I} K_i)$. Since we also have that $x=\sum_{i\in[m]\setminus I}\frac{-\lambda_i}{\Lambda}x_i\in\conv(\bigcup_{i\in[m]\setminus I} K_i)$, it follows that $x\in\conv(\bigcup_{i\in I} K_i)\cap\conv(\bigcup_{i\in[m]\setminus I} K_i)\neq\emptyset$. This completes the proof.
\end{proof}

Let $K_1,\ldots,K_n\subset\R^n$ be some strongly separated convex bodies and $v_i\in K_i$ for each $i\in[n]$. We consider the linear polynomial $$l_v=\det\begin{pmatrix} 1&\ldots&1&1\\v_1&\ldots&v_n&x \end{pmatrix}\in\R[x_1,\ldots,x_n], \textrm{ where }v=(v_1,\ldots,v_n).$$

\begin{prop}\label{prop:convex}
 There are some $v_i\in K_i$ for $i\in[n]$ such that $l_v$ is nonnegative (or nonpositive) on every $K_i$. If $w_i\in K_i$ for $i\in[n]$ is another choice of points such that $l_w$ is nonnegative (resp. nonpositive) on every $K_i$, then $l_v=\lambda\cdot l_w$ for some $\lambda>0$.
\end{prop}

\begin{proof}
 We will only show the claims regarding nonnegativity. The proof for nonpositivity is almost verbatim the same.

 We proceed by induction on $n$. The claim is clear for $n=1$, so we consider $n>1$. To any point $u_n\in K_n$ we will associate a linear polynomial in the following way. By the Hahn--Banach separation theorem, there exists an affine hyperplane $H$ in $\R^n$ which separates $K_n$ and $\conv(\bigcup_{i=1}^{n-1}K_i)$. The linear projection $\pi$ from $u_n$ to $H$ maps $K_1,\ldots,K_{n-1}$ to the convex bodies $\pi(K_1),\ldots,\pi(K_{n-1})$ which are strongly separated.
 By induction hypothesis there are such $u_i'\in\pi(K_i)$ that lift to some $u_i\in K_i$ and the corresponding linear polynomial $l_u$ is nonnegative on $K_1,\ldots,K_{n-1}$. We are therefore done if we can choose $u_n\in K_n$ such that $l_u$ is nonnegative on $K_n$. To that end we define the sets $V(u_n)=\{x\in K_n:\,l_u(x)\geq0\}$ and $V^\mathrm{o}(u_n)=\{x\in K^\mathrm{o}_n:\,l_u(x)>0\}$ respectively where $K^\mathrm{o}_n$ denotes the (nonempty) interior of $K_n$ in $\mathbb{R}^n$. Now let $\tilde{u}_n\in K_n$ be another point. As above we obtain the linear polynomial $l_{\tilde{u}}$ which is nonnegative on $K_1,\ldots,K_{n-1}$. If $l_u(\tilde{u}_n)=0$, then $l_{\tilde{u}}=\lambda l_u$ for some $\lambda>0$. Otherwise, by the uniqueness of the affine hyperplane $l_{\tilde{u}}=0$, there is no point $x\in K_n$ where both $l_u$ and $l_{\tilde{u}}$ vanish. This implies that $V(\tilde{u}_n)\subset V(u_n)$ if and only if $\tilde{u}_n \in V(u_n)$ and $V^\mathrm{o}(\tilde{u}_n)\subsetneq V^\mathrm{o}({u}_n)$ if and only if $\tilde{u}_n \in V^\mathrm{o}(u_n)$. One consequence is that $\{V^\mathrm{o}(x)\}_{x\in K^\mathrm{o}_n}$ is an open cover of $K^\mathrm{o}_n$. By Lindel\"of's lemma, there exists a sequence $\{x_m\}_{m\in\N}$ in $K_n^\mathrm{o}$ such that $\{V^\mathrm{o}(x_m)\}_{m\in\N}$ is again a cover of $K_n^\mathrm{o}$. After passing to a suitable subsequence, we can furthermore assume that $V^\mathrm{o}(x_m)\subsetneq V^\mathrm{o}(x_{m'})$ if $m<m'$ and that the sequence $(x_m)_{m\in\N}$ converges to some $v_n\in K_n$. We claim that $V(v_n)=K_n$ which gives our desired linear polynomial $l_v$. Indeed, if $V(v_n)\subsetneq K_n$, then there is a $w_n\in K^\mathrm{o}_n \setminus V(v_n)$ because the closure of $K^\mathrm{o}_n$ is $K_n$ and $V(v_n)$ is closed. Then, since $w_n\in V^\mathrm{o}(x_m)$ for sufficiently large $m$, we have $x_m\not\in V^\mathrm{o}(w_n)$ and $v_n\in V^\mathrm{o}(w_n)$ for sufficiently large $m$. This contradicts the fact that $v_n$ is the limit of the $x_m$.
 
 
 Now let $w_i\in K_i$ for $i\in[n]$ be another choice of points such that $l_w$ is nonnegative on every $K_i$ and assume that the hyperplanes defined by $l_v$ and $l_w$ are different. It follows from the induction hypothesis that neither $l_v$ nor $l_w$ vanish on the line spanned by $v_n$ and $w_n$. For every $t\in]0,1[$ we consider $u_{n,t}=(1-t)\cdot v_n+t\cdot w_n$ and the linear polynomial $l_{u_t}$ obtained as above. Since $l_v(w_n)>0$, we also have $l_{u_t}(w_n)>0$ for small values of $t>0$. But since $l_{u_t}$ has a zero on the line segment between $v_n$ and $w_n$, we have $l_{u_t}(v_n)<0$ for small $t>0$. On the other hand we have $l_{u_t}(v_n)>0$ for $t$ close to $1$ since $l_{w}(v_n)>0$. Therefore, there is a $t_0\in]0,1[$ such that $l_{u_{t_0}}(v_n)=0$. But then the induction hypothesis implies that $l_v=\lambda\cdot l_{u_{t_0}}$ for some $\lambda>0$.  Thus $l_v$ vanishes at $v_n$ and at $u_{n,t_0}$. This implies that $l_v$ also vanishes at $w_n$ since $w_n$ lies on the line spanned by $v_n$ and $u_{n,t_0}$. But this we have already excluded above. Thus the hyperplanes defined by $l_v$ and $l_w$ are not different.
\end{proof}

\begin{cor}\label{cor:thx}
 The convex hull of strongly separated $K_1,\ldots,K_n\subset\R^n$ has at least two faces of dimension $n-1$.
\end{cor}

\begin{proof}
 Take some $v_i\in K_i$ for $i\in[n]$ such that $l_v$ is nonnegative on every $K_i$. The affine hyperplane defined by $l_v$ supports the convex hull of $K_1,\ldots,K_n$ and is spanned by $v_1,\ldots,v_n$. This gives a face of dimension $n-1$. Another such face is obtained from such a linear form $l_{v'}$ that is nonpositive on all $K_i$.
\end{proof}

\section{Canonically embedded curves}
In this section we let $X\subset\pp^{g-1}$ always denote a non-hyperelliptic (smooth geometrically irreducible projective) curve of genus $g\geq3$ embedded via the canonical embedding.
\begin{Def}
 The \textit{avoidance locus} of $X$ is the set of real hyperplanes in $\pp^{g-1}$ that do not intersect $X(\R)$. It is a subset of the dual projective space $(\pp^{g-1})^\vee$.
\end{Def}
Real points of the avoidance locus correspond to definite differentials having no real zero. Thus it is a direct consequence of Corollary \ref{cor:avoidloc} that the avoidance locus of $X$ has $(2^{s-1}-1+a)$ connected components. Let us take a hyperplane $H$ from the avoidance locus defined by a definite differential $\omega\in\Gamma(X,\Omega_{X/\R})$ and look at the corresponding affine chart $\pp^{g-1}\setminus H$. In this affine chart we have that $X(\R)\subset\R^{g-1}$ is a compact subset. We are interested in the convex hull $K$ of $X(\R)$. The linear polynomials that cut out $K$ correspond to definite differentials with the same sign pattern as our fixed $\omega$. All $(g-2)$-dimensional faces of $K$ correspond to such definite differentials  which have $g-1$ affinely independent intersection points with $X(\R)$. These come from nonvanishing real odd theta characteristics. Therefore, Corollary \ref{cor:countodd} gives us the following upper bound:

\begin{thm}
 The convex hull of a compact canonical curve $X(\R)\subset\R^{g-1}$ of genus $g$ has at most $2^{g-1}$ faces of dimension $g-2$.
\end{thm}

\begin{rem}
 The upper bound from the previous theorem is attained if all real odd theta characteristics are nonvanishing and totally real.
\end{rem}

\begin{ex}
 It follows that the convex hull of a space sextic ($g=4$) can have at most eight $2$-faces. On the other hand the sextic curve constructed in \cite[Example 2.2]{kulkarni2017real} has all its odd theta characteristics totally real, so its convex hull has indeed eight $2$-faces for every suitable affine chart. This answers \cite[Question 11, page 254]{kulkarni2017real}.
\end{ex}

Using the results from the previous section we also obtain some lower bounds.
\begin{thm}
 Let $K$ be the convex hull of a compact canonical curve $X(\R)\subset\R^{g-1}$ of genus $g$. Then we have the following:
 \begin{enumerate}
  \item If $s=g-1$ and $X$ is not of dividing type, then $K$ has at least two faces of dimension $g-2$.
  \item If $s= g$, then $K$ has at least $g$ faces of dimension $g-2$.
 \end{enumerate}
\end{thm}

\begin{proof}
 It follows from Corollary \ref{cor:avoidloc} that the convex hulls the components of $X(\R)$ are strongly separated. Thanks to Corollary \ref{cor:thx}, for $s=g-1$ we get at least two $(g-2)$-faces. For $s=g$ we get one $(g-2)$-face of $K$ for each choice of $g-1$ components.
\end{proof}

In the case of $M$-curves, i.e. $s=g+1$, the situation is slightly more complicated.

\begin{thm}
 Let $K$ be the convex hull of a compact canonical curve $X(\R)\subset\R^{g-1}$ of genus $g$ with $s=g+1$ components (thus $X$ is of dividing type). Let $k$ be the number of ovals where the orientation induced by the differential $\omega_\infty$ corresponding to the hyperplane at infinity agrees with the complex orientation. Then the number of $(g-2)$-faces is at least $k\cdot(g+1-k)$. In particular, we always have at least $g$ such faces.
\end{thm}

\begin{figure}[h]
 \includegraphics[width=5cm]{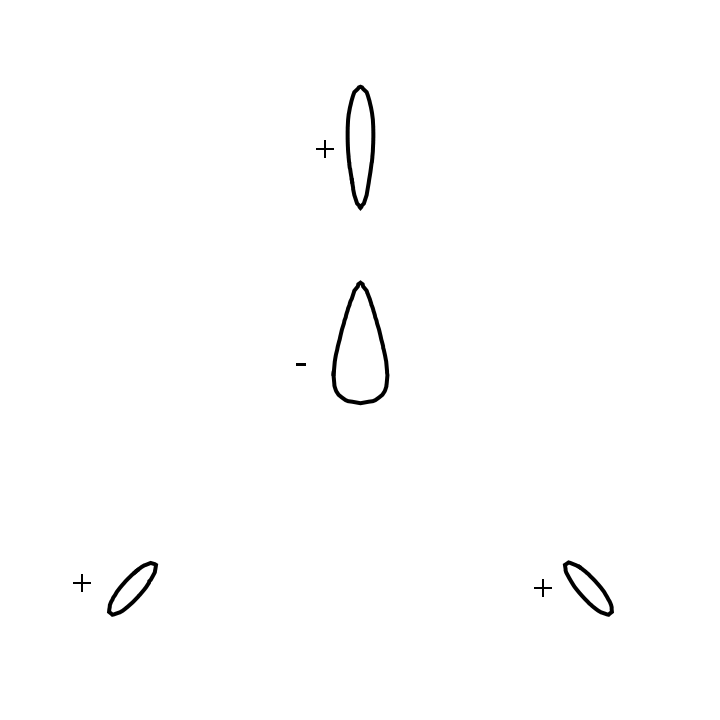} \quad
 \includegraphics[width=5cm]{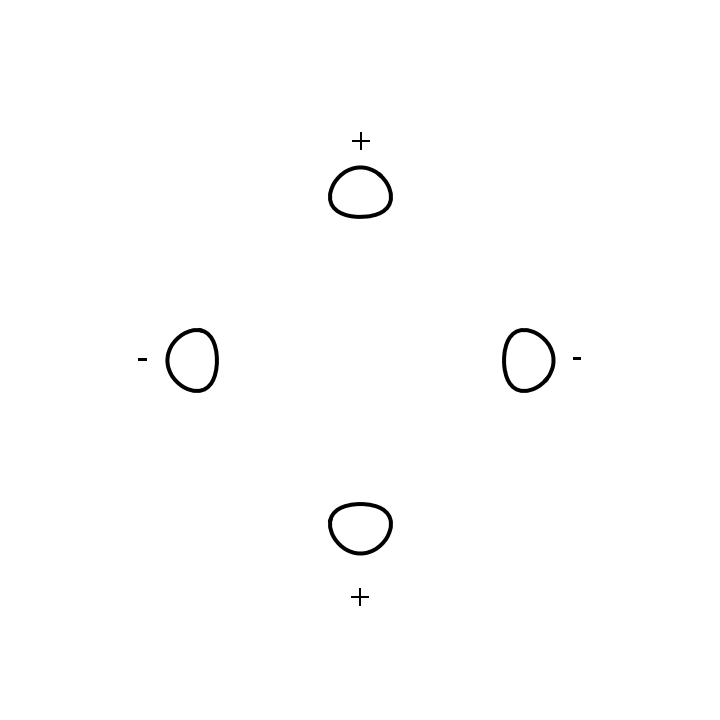}
\caption{Two canonically embedded curves of genus three: The signs indicate whether the orientation induced by the hyperplane at infinity does or does not agree with the complex orientation.}
\label{fig:quarticplanecurves}
\end{figure}

\begin{proof}
 For every choice $X_{i_1},\ldots,X_{i_{g-1}}$ of $g-1$ ovals we obtain at least two $(g-2)$-faces of their convex hull $K'$ from Proposition \ref{prop:convex}. We claim that one of these is a face of $K$ if and only if $\omega_\infty$ induces the complex orientation on one of the remaining ovals and not on the other. This would imply our claim.
 
 First assume that one of these two faces is a face of $K$ as well. The affine span of each choice of points $x_1,\ldots,x_{g-1}$ such that $x_j\in X_{i_j}$ does not intersect any of the two remaining ovals $X_{i_{g}}$ and $X_{i_{g+1}}$. Therefore, a linear polynomial defining such an affine hyperplane is always nonnegative (or nonpositive) on the remaining two ovals. Applying Proposition \ref{prop:convex} to some other affine charts then shows that there are definite differentials inducing the same orientation on $X_{i_{g}}$ and $X_{i_{g+1}}$ as $\omega_\infty$ and every other combination of orientations on $X_{i_1},\ldots,X_{i_{g-1}}$. But this can only be the case if $\omega_\infty$ does not agree with the complex orientation on both $X_{i_{g}}$ and $X_{i_{g+1}}$ by Proposition \ref{prop:vincomp}.
 
 Now assume that $\omega_\infty$ induces the complex orientation on one of the remaining ovals and not on the other. Then by Proposition \ref{prop:vincomp} there is a linear polynomial $l_\epsilon$ for every choice of signs $\epsilon_1,\ldots,\epsilon_{g-1}\in\{\pm 1\}$ that is nonnegative on $X_{i_{g}}$ and $X_{i_{g+1}}$ such that $\epsilon_j l_\epsilon$ is nonnegative on $X_{i_j}$. From these we can construct a linear polynomial that vanishes on at least one point of each $X_{i_1},\ldots,X_{i_{g-1}}$ and is nonnegative on $X_{i_{g}}$ and $X_{i_{g+1}}$. But this implies that one of the two $(g-2)$-faces of $K'$ obtained from Proposition \ref{prop:convex} is actually a face of $K$ as well since no affine hyperplane that intersects all $X_{i_1},\ldots,X_{i_{g-1}}$ can intersect $X_{i_{g}}$ or $X_{i_{g+1}}$.
\end{proof}

Summing up the lower bounds from the two preceding theorems over all components of the avoidance locus we obtain the following lower bound for totally real theta characteristics.

\begin{cor}
 Assume that $X$ is not of dividing type or that $X$ is an $M$-curve. Then there are at least $\binom{s}{g-1}\cdot 2^{g-1}$ totally real odd theta characteristics.
\end{cor}

\begin{rem}
 In this section we have only considered non-hyperelliptic curves since the convexity questions do not make sense otherwise. However the task of calculating totally real theta characteristics still makes sense and is in fact not hard to carry out. Indeed, let $X$ be a hyperelliptic curve of genus $g\geq2$ with $X(\R)\neq\emptyset$. Then there is a $2$-to-$1$ map $\pi$ from $X$ to the projective line defined over the reals \cite[Prop. 6.1]{grossharris}. We can therefore represent $X$ by the equation $$y^2=\prod_{i=1}^{2r} (x-p_i) \cdot \prod_{i=r+1}^{g+1-r}(x-q_i)(x-\overline{q_i})$$ where the $p_i$ are real and the $q_i$ are non-real complex numbers and $0\leq r\leq g+1$. If $1\leq r\leq g$, then $X(\R)$ has $s=r$ components and $X$ is not of dividing type. If $r=g+1$, then $X$ is an $M$-curve and thus of dividing type. If $r=0$, then $X$ is of dividing type with $s=1$ or $s=2$ depending on the parity of $g$. Now let $D$ be a divisor on $X$ from the linear system associated with $\pi$. Clearly, we can choose $D$ to be totally real. The theta characteristics of $X_\C$ can all be written as $$M=(k-1)\cdot D+E$$ where $0\leq k\leq\frac{g+1}{2}$ and $E$ is a formal sum of elements in a subset of the $2g+2$ branch points of $\pi$. Furthermore, one has $l(M)=k$. It follows that $$\sum_{l=0}^{\left \lfloor{\frac{g-1}{4}}\right \rfloor}\binom{2r}{g-1-4l}$$ is the number of totally real odd theta characteristics of $X$. In particular, if $g$ is relatively large compared to $r$,  there are no totally real odd theta characteristics.
\end{rem}

We conclude with an examination of the tightness of our bounds.

\begin{ex}
 All our bounds, upper and lower, are attained in the case of canonically embedded curves of genus three, i.e. plane quartic curves. In Table \ref{tab:quartic} we list these bounds for the number of totally real theta characteristics and the number of edges of (the convex hull of) a canonically embedded curve of genus three for all combinations of $s$, i.e. the number of ovals, and $a$, i.e. whether the curve is of dividing type or not. In the last column we give explicit examples realizing these bounds.
\end{ex}

\begin{table}[h]
\begin{tabular}{|c|c|c|c|c|}
\hline
 $s$ & $a$ & T & E & Examples \\ \hline
 1&1&0&0& $1 - x^4 - y^4$\\ &&4&4& $1 + 12 x^2 - 10 x^4 + 12 y^2 - 101 x^2 y^2 - 10 y^4$ \\ \hline
 2&1&4&2&$9 - 10 x^2 + 2 x^4 + 6 y^2 + 2 x^2 y^2 + 2 y^4$ \\ &&8&4&$15 - 92 x^2 + 80 x^4 + 8 x y - 92 y^2 + 416 x^2 y^2 + 80 y^4$ \\ \hline
 2&0&0&0&$201 - 300 x^2 + 101 x^4 - 300 y^2 + 200 x^2 y^2 + 101 y^4$ \\ &&4&4&$3 - 24 x^2 + 19 x^4 - 24 y^2 + 104 x^2 y^2 + 19 y^4$ \\ \hline
 3&1&12&3&$39 - 155 x^2 + 200 x^4 + 10 x y - 20 x^2 y - 155 y^2 - 40 x y^2$\\&&&&$ + 
 260 x^2 y^2 - 20 y^3 + 40 x y^3 + 220 y^4$ \\ &&16&4&$101 + 2 x - 600 x^2 + 500 x^4 + 2 y - 600 y^2 + 2600 x^2 y^2 + 500 y^4$ \\ \hline
 4&0&24&3& $-1 + 80 x^2 + 80 x^4 + 240 x^2 y + 20 y^2 - 40 y^3 + 20 y^4$\\ &&28&4&$101 - 600 x^2 + 501 x^4 - 600 y^2 + 2600 x^2 y^2 + 501 y^4$ \\ \hline
\end{tabular}
\caption{Examples realizing our derived lower and upper bounds for the number of totally real theta characteristics (T) and edges (E) of the convex hull.}\label{tab:quartic}
\end{table}

\begin{ex}
 In the case of canonically embedded curves of genus four, i.e. space sextics, all of our bounds on the number of totally real theta characteristics turn out to be tight, except for the lower bound for $M$-curves. In Table \ref{tab:sextic} we list these bounds and compare them with the numbers from \cite[Table 1]{kulkarni2017real} and \cite[Table 1]{hauentri}. 
 Furthermore, one can extract from \cite[Table 1]{kulkarni2017real} and \cite[Table 1]{hauentri} that all our bounds on the number of $2$-faces of the convex hull of a space sextic are realized. For example if we take a space sextic with $s=5$, $a=0$ and only $84$ totally real theta characteristics, then there will be a connected component of the avoidance locus that contains only four totally real theta characteristics. Indeed, there are ten components that contain at least six totally real theta characteristics, so the remaining five components cannot all contain more than four totally real theta characteristics. Taking the hyperplane at infinity from such a component gives a compact curve in $\R^3$ whose convex hull has only four $2$-faces. In all the other cases we can argue analogously.
\end{ex}
 
\begin{table}[h]
\begin{tabular}{|c|c|c|c|c|}
\hline
 $s$ & $a$ & Our bounds  & Realized in \cite{kulkarni2017real} or \cite{hauentri} \\ \hline
 1&0&[0,0]&[0,0] \\ 
 1&1&[0,8]&[0,8]\\ 
 2&1&[0,16]&[0,16]\\ 
 3&0&[0,24]&[0,24]\\ 
 3&1&[8,32]&[8,32]\\ 
 4&1&[32,64]&[32,64]\\ 
 5&0&[\textbf{80},120]&[\textbf{84},120]\\  \hline
\end{tabular}
\caption{Comparing our bounds on the number of totally real theta characteristics with the numbers realized in \cite{kulkarni2017real} or \cite{hauentri}.}\label{tab:sextic}
\end{table}

\bigskip

 \noindent \textbf{Acknowledgements.}
I would like to thank all the participants of the ``Tritangent Summit'' that took place at the Max Planck Institute for Mathematics in the Sciences in Leipzig in January 2018. Furthermore, I want to thank the diligent referee who pointed out some inaccuracies and gaps and even helped to eliminate them.

\bibliographystyle{alpha}
\bibliography{biblio}
 \end{document}